\documentclass[11pt]{amsart}

\usepackage{lscape}
\usepackage[margin=1in]{geometry}

\newtheorem{theorem}{Theorem}[section]

\newtheorem{prop}[theorem]{Proposition}

\theoremstyle{definition}

\theoremstyle{remark}
\newtheorem{remark}[theorem]{Remark}

\numberwithin{equation}{section}

\begin{document}

\newcommand{\spacing}[1]{\renewcommand{\baselinestretch}{#1}\large\normalsize}
\spacing{1.14}

\title{On the left invariant Randers and Matsumoto metrics of Berwald type on 3-dimensional Lie groups}

\author {H. R. Salimi Moghaddam}

\address{Department of Mathematics\\ Faculty of  Sciences\\ University of Isfahan \\ Isfahan\\ 81746-73441-Iran.}
\email{salimi.moghaddam@gmail.com and hr.salimi@sci.ui.ac.ir}



\begin{abstract}
In this paper we identify all simply connected 3-dimensional real Lie groups which admit Randers or Matsumoto metrics of Berwald type with a certain underlying left invariant Riemannian metric. Then we give their flag curvatures formulas explicitly.
\end{abstract}

\maketitle
\section{\textbf{Introduction}}\label{Intro}
Let $M$ be a smooth $n-$dimensional manifold and $TM$ be its
tangent bundle. A Finsler metric on $M$ is a non-negative function
$F:TM\longrightarrow \Bbb{R}$ which has the following properties:
\begin{enumerate}
    \item $F$ is smooth on the slit tangent bundle
    $TM^0:=TM\setminus\{0\}$,
    \item $F(x,\lambda y)=\lambda F(x,y)$ for any $x\in M$, $y\in T_xM$ and $\lambda
    >0$,
    \item the $n\times n$ Hessian matrix $[g_{ij}(x,y)]=[\frac{1}{2}\frac{\partial^2 F^2}{\partial y^i\partial
    y^j}]$ is positive definite at every point $(x,y)\in TM^0$.
\end{enumerate}
An important family of Finsler metrics is the family of
$(\alpha,\beta)-$metrics. These metrics are introduced by M.
Matsumoto (see \cite{Ma2}). Two families of interesting and important examples
of $(\alpha,\beta)-$metrics are Randers metrics $\alpha+\beta$,
and Matsumoto metrics $\frac{\alpha^2}{\alpha-\beta}$, where
$\alpha(x,y)=\sqrt{g_{ij}(x)y^iy^j}$ and $\beta(x,y)=b_i(x)y^i$
and $g$ and $\beta$ are a Riemannian metric and  a 1-form respectively
as follows:
\begin{eqnarray}
  g&=&g_{ij}dx^i\otimes dx^j \\
  \beta&=&b_idx^i.
\end{eqnarray}
These metrics occur naturally in physical applications. For example, in four-dimensional case, consider the above Randers metric as
\begin{equation}\label{electric charge test}
    F(x,\dot{x})=(g_{ij}(x)\dot{x}^i\dot{x}^j)^{\frac{1}{2}}+b_i(x)\dot{x}^i,
\end{equation}
where $\dot{x}^i=\frac{dx^i}{dq}$ and $q$ is a parameter along a curve. Now suppose that $g_{ij}(x)$ is the \\
(pseudo-)Riemannian metric tensor of space-time and
\begin{equation}\label{1-form}
    b_i(x)=\frac{e}{m_0c^2}A_i(x),
\end{equation}
where $e$ and $m_0$ are the electric charge and the rest mass of a test particle respectively, $c$ is the light velocity, and $A_i(x)$ denotes the electromagnetic vector potential. In this case the expression (\ref{electric charge test}) is, up to a constant factor, the Lagrangian function of a test electric charge in the electromagnetic and gravitational fields described by the vector potential $A_i(x)$ and the Riemannian metric tensor
$g_{ij}(x)$ (see \cite{As}.).\\
Randers metrics were introduced by G. Randers in 1941 (\cite{Ra}) when he was working on general relativity. On the other hand in the Matsumoto metric, the 1-form $\beta=b_idx^i$ was originally to be induced by earth's gravity (see \cite{AnInMa} or \cite{Ma1}.).\\

It has been shown that $\alpha+\beta$  (and in a similar way $\frac{\alpha^2}{\alpha-\beta}$) is
a Finsler metric if and only if $\|\beta_x\|_\alpha=\sqrt{b_i(x)b^i(x)}<1$, (for $\frac{\alpha^2}{\alpha-\beta}$, $\|\beta_x\|_\alpha=\sqrt{b_i(x)b^i(x)}<\frac{1}{2}$), where
$b^i(x)=g^{ij}(x)b_j(x)$ and $[g^{ij}(x)]$ is the inverse matrix
of $[g_{ij}(x)]$ (for more details see \cite{AnDe} and \cite{ChSh}).\\
The Riemannian metric $g$ induces an inner
product on any cotangent space $T^\ast_xM$ such that
$<dx^i(x),dx^j(x)>=g^{ij}(x)$. The induced inner product on
$T^\ast_xM$ induces a linear isomorphism between $T^\ast_xM$ and
$T_xM$ (see \cite{DeHo}.). Then the 1-form $\beta$
corresponds to a vector field $\tilde{X}$ on $M$ such that
\begin{eqnarray}
  g(y,\tilde{X}(x))=\beta(x,y).
\end{eqnarray}

Therefore we can write the Randers and Matsumoto metrics as follows:
\begin{eqnarray}
  F(x,y)&=&\alpha(x,y)+g(\tilde{X}(x),y),\label{Randers}\\
  F(x,y)&=&\frac{\alpha(x,y)^2}{\alpha(x,y)-g(\tilde{X}(x),y)}\label{Matsumoto}.
\end{eqnarray}

One of geometric quantities which can be used in classification of Finsler metrics is flag curvature. This quantity is a generalization of the concept of sectional curvature in Riemannian geometry. Flag
curvature is computed by the following formula:
\begin{eqnarray}\label{flag curvature}
  K(P,Y)=\frac{g_Y(R(U,Y)Y,U)}{g_Y(Y,Y).g_Y(U,U)-g_Y^2(Y,U)},
\end{eqnarray}
where $g_Y(U,V)=\frac{1}{2}\frac{\partial^2}{\partial s\partial
t}(F^2(Y+sU+tV))|_{s=t=0}$, $P=span\{U,Y\}$,
$R(U,Y)Y=\nabla_U\nabla_YY-\nabla_Y\nabla_UY-\nabla_{[U,Y]}Y$ and
$\nabla$ is the Chern connection induced by $F$ (see \cite{BaChSh}
and \cite{Sh}.).

A Riemannian Metric $g$ on the Lie group $G$ is called left
invariant if
\begin{eqnarray}
  g(x)(y,z)=g(e)(T_xl_{x^{-1}}y,T_xl_{x^{-1}}z) \ \ \ \ \forall x\in
  G, \forall y,z\in T_xG,
\end{eqnarray}
where $e$ is the unit element of $G$.\\
Similar to the Riemannian case, a Finsler metric is called left
invariant if
\begin{eqnarray}
  F(x,y)=F(e,T_xl_{x^{-1}}y).
\end{eqnarray}
In our previous paper (\cite{Sa3}) we studied invariant Matsumoto metrics on homogeneous spaces and gave the flag curvature formula of them.
In this paper we identify all simply connected 3-dimensional real Lie groups which admit Randers or Matsumoto metrics of
Berwald type with a certain underlying left invariant Riemannian metric. Then we give their flag curvatures formulas explicitly.

\section{\textbf{Left invariant Randers and Matsumoto metrics of Berwald type on 3-dimensional Lie groups}}\label{Main}
The following proposition gives a method for constructing left invariant Randers or Matsumoto metrics on Lie groups.

\begin{prop}\label{existence}
Let $<,>$ be any left invariant Riemannian metric on a Lie group
$G$. A function $F$ of the form (\ref{Randers}) or (\ref{Matsumoto}) is a left invariant Finsler metric
(Randers or Matsumoto metric, respectively) if and only if the vector field $\tilde{X}$ is a left
invariant vector field and, $\|\tilde{X}(x)\|_\alpha=\sqrt{g(\tilde{X}(x),\tilde{X}(x))}<1$ (for Randers metric) or $\|\tilde{X}(x)\|_\alpha<\frac{1}{2}$ (for Matsumoto metric), for any $x\in G$.
\end{prop}

\begin{proof}
With attention to proposition 3.6 of \cite{Sa2}, it is clear.
\end{proof}

\begin{theorem}\label{Classification}
Let $G$ be a simply connected 3-dimensional real Lie group equipped with a left invariant Riemannian metric $g$ and $\frak{g}$ denotes the Lie algebra of $G$. Then $G$ admits a left invariant Randers or Matsumoto metrics of Berwald type with the underlying  Riemannian metric $g$ if and only if $G$, $\frak{g}$ and the associated matrix of $g$, with respect to the base $\{x,y,z\}$ of $\frak{g}$ and up to automorphism,  be as follows:

\begin{description}
  \item[i] $G=\Bbb{R}^3$, $[x,y]=0,$ $[x,z]=0,$ $[y,z]=0$, $\left(
                                                                         \begin{array}{ccc}
                                                                           1 & 0 & 0 \\
                                                                           0 & 1 & 0 \\
                                                                           0 & 0 & 1 \\
                                                                         \end{array}
                                                           \right)$,\\
  \item[ii] $G=$ \mbox{The solvable Lie group} $\tilde{E}_0(2)$, $[x,y]=0,$ $[x,z]=y,$ $[y,z]=-x$, $\left(
                                                                         \begin{array}{ccc}
                                                                           1 & 0 & 0 \\
                                                                           0 & 1 & 0 \\
                                                                           0 & 0 & \nu \\
                                                                         \end{array}
                                                        \right)$, where $\nu>0$,\\
  \item[iii] $G=$ \mbox{The non-unimodular Lie group} $G_c$, $[x,y]=0,$ $[x,z]=-y,$ $[y,z]=cx-2y$, $\left(
                                                                         \begin{array}{ccc}
                                                                           1 & \frac{1}{2} & 0 \\
                                                                           \frac{1}{2} & 1 & 0 \\
                                                                           0 & 0 & \nu \\
                                                                         \end{array}
                                                        \right)$, where $\nu>0 ,$ $c=0$.
\end{description}
\end{theorem}

\begin{proof}
We consider the classification of simply connected 3-dimensional real Lie groups given in \cite{HaLee}. Then, by some computations, we can compute  the Levi-Civita connection of each case with respect to the base $\{x,y,z\}$ as table (\ref{Levi-Civita connection}), where for case 15 we assumed that \begin{eqnarray}
A&:=&\frac{\lambda^2(1+\mu)+1-\mu}{2c^2\lambda^2},\\
B&:=&\frac{1-\mu}{2c\lambda^2},\\
D&:=&\frac{1-\mu}{2\lambda^2}.
\end{eqnarray}
We know that (see \cite{BaChSh} and \cite{AnInMa}.) a Randers or Matsumoto metric is of Berwald type if and only if the vector field $\tilde{X}$ (considered in (\ref{Randers}) and (\ref{Matsumoto})) is parallel with respect to the Levi-Civita connection of the Riemannian metric $g$ and, $\|\tilde{X}\|_\alpha<1$ (for Randers metric) or $\|\tilde{X}\|_\alpha<\frac{1}{2}$ (for Matsumoto metric). Therefore it is sufficient to find the cases which admit a parallel left invariant vector field $\tilde{X}$ between these $15$ cases. We can see in case $1$ every left invariant vector field $\tilde{X}$ is parallel with respect to the Levi-Civita connection of $g$, so for any left invariant vector field $\tilde{X}$ with $\|\tilde{X}\|_\alpha<1$ we have a Randers metric (and with $\|\tilde{X}\|_\alpha<\frac{1}{2}$ a Matsumoto metric). By a direct computation and using table (\ref{Levi-Civita connection}) in case $5$, $G$ admits such left invariant vector field $\tilde{X}$ if and only if $\mu=1$. In this case the left invariant vector fields of the form $\tilde{X}=pz$, $p\in\Bbb{R}$, are parallel with respect to the Levi-Civita connection of $g$. Therefore, in Randers case, it is sufficient to consider $\|\tilde{X}\|_\alpha=\sqrt{g(pz,pz)}=|p|\sqrt{\nu}<1$ or $|p|<\frac{1}{\sqrt{\nu}}$. 
\fontsize{6}{0}{\selectfont
\begin{landscape}
\begin{table}
    \centering\caption{The Levi-Civita connection of simply connected 3-dimensional real Lie groups}\label{Levi-Civita connection}
        \begin{tabular}{|p{0.5cm}|p{1.5cm}|p{1.5cm}|p{2cm}|p{1cm}|p{1cm}|p{1cm}|p{1.5cm}|p{1cm}|p{1cm}|p{2.5cm}|p{1.5cm}|p{1.5cm}|p{0.5cm}|}
        \hline
            case & Lie algebra structure & Associated  simply connected  Lie group & Left invariant Riemannian metric & Conditions for  parameters & \hspace{3cm}$\nabla_xx$ & \hspace{3cm}$\nabla_xy$ & \hspace{3cm}$\nabla_xz$ & \hspace{3cm}$\nabla_yx$ & \hspace{3cm}$\nabla_yy$ & \hspace{3cm}$\nabla_yz$ & \hspace{3cm}$\nabla_zx$ & \hspace{3cm}$\nabla_zy$ & \hspace{3cm}$\nabla_zz$ \\
             \hline
           case 1 & $[x,y]=0,$ $[x,z]=0,$ $[y,z]=0$ & $\Bbb{R}^3$ & $\left(
                                                                         \begin{array}{ccc}
                                                                           1 & 0 & 0 \\
                                                                           0 & 1 & 0 \\
                                                                           0 & 0 & 1 \\
                                                                         \end{array}
                                                                       \right)$ & - & 0 & 0 & 0 & 0 & 0 & 0 & 0 & 0 & 0  \\
            \hline
           case 2 & $[x,y]=z,$ $[x,z]=0,$ $[y,z]=0$ & The Heisenberg group $Nil$ & $\left(
                                                                         \begin{array}{ccc}
                                                                           \lambda & 0 & 0 \\
                                                                           0 & \lambda & 0 \\
                                                                           0 & 0 & 1 \\
                                                                         \end{array}
                                                                       \right)$ & $\lambda>0$ & 0 & $\frac{1}{2}z$ & $\frac{-1}{2\lambda}y$ & $-\frac{1}{2} z$ & 0 & $\frac{1}{2\lambda}x$ & $-\frac{1}{2\lambda}y$ & $\frac{1}{2\lambda}x$ & 0  \\
            \hline
            case 3 & $[x,y]=0,$ $[x,z]=-x,$ $[y,z]=y$ & The solvable Lie group $Sol$ & $\left(
                                                                         \begin{array}{ccc}
                                                                           1 & 0 & 0 \\
                                                                           0 & 1 & 0 \\
                                                                           0 & 0 & \nu \\
                                                                         \end{array}
                                                                       \right)$ & $\nu>0$ & $\frac{z}{\nu}$ & 0 & $-x$ & 0 & -$\frac{z}{\nu}$ & $y$ & 0 & 0 & 0  \\
            \hline
            case 4 & $[x,y]=0,$ $[x,z]=-x,$ $[y,z]=y$ & The solvable Lie group $Sol$ & $\left(
                                                                         \begin{array}{ccc}
                                                                           1 & 1 & 0 \\
                                                                           1 & \mu & 0 \\
                                                                           0 & 0 & \nu \\
                                                                         \end{array}
                                                                       \right)$ & $\mu>1 , \nu>0$ & $\frac{z}{\nu}$ & 0 & $\frac{\mu}{1-\mu}x-\frac{1}{1-\mu}y$ & 0 & -$\frac{\mu}{\nu}z$ & $\frac{\mu}{1-\mu}x-\frac{\mu}{1-\mu}y$ & $\frac{1}{1-\mu}x-\frac{1}{1-\mu}y$ & $\frac{\mu}{1-\mu}x-\frac{1}{1-\mu}y$ & 0  \\
            \hline
            case 5 & $[x,y]=0,$ $[x,z]=y,$ $[y,z]=-x$ & The solvable Lie group $\tilde{E}_0(2)$ & $\left(
                                                                         \begin{array}{ccc}
                                                                           1 & 0 & 0 \\
                                                                           0 & \mu & 0 \\
                                                                           0 & 0 & \nu \\
                                                                         \end{array}
                                                                       \right)$ & $0<\mu\leq1$ , $\nu>0$ & 0 & $\frac{1-\mu}{2\nu}z$ & $\frac{\mu-1}{2\mu}y$ & $\frac{1-\mu}{2\nu}z$ & 0 & $\frac{\mu-1}{2}x$ & -$\frac{1+\mu}{2\mu}y$ & $\frac{1+\mu}{2}x$ & 0 \\
            \hline
            case 6 & $[x,y]=2z,$ $[x,z]=-2y,$ $[y,z]=-2x$ & The simple Lie group $\widetilde{PSL}(2,\Bbb{R})$ & $\left(
                                                                         \begin{array}{ccc}
                                                                           \lambda & 0 & 0 \\
                                                                           0 & \mu & 0 \\
                                                                           0 & 0 & \nu \\
                                                                         \end{array}
                                                                       \right)$ & $\mu\geq\nu>0$ , $\lambda>0$ & 0 & $\frac{\lambda+\mu+\nu}{\nu}z$ & $-\frac{\lambda+\mu+\nu}{\mu}y$ & $\frac{\lambda+\mu-\nu}{\nu}z$ & 0 & $\frac{-\lambda-\mu+\nu}{\lambda}x$ & $\frac{-\lambda+\mu-\nu}{\mu}y$ & $\frac{\lambda-\mu+\nu}{\lambda}x$ & 0 \\
            \hline
            case 7 & $[x,y]=z,$ $[x,z]=-y,$ $[y,z]=x$ & The simple Lie group $SU(2)$ & $\left(
                                                                         \begin{array}{ccc}
                                                                           \lambda & 0 & 0 \\
                                                                           0 & \mu & 0 \\
                                                                           0 & 0 & \nu \\
                                                                         \end{array}
                                                                       \right)$ & $\lambda\geq\mu\geq\nu>0$ & 0 & $\frac{-\lambda+\mu+\nu}{2\nu}z$ & $\frac{\lambda-\mu-\nu}{2\mu}y$ & $\frac{-\lambda+\mu-\nu}{2\nu}z$ & 0 & $\frac{\lambda-\mu+\nu}{2\lambda}x$ & $\frac{\lambda+\mu-\nu}{2\mu}y$ & $\frac{-\lambda-\mu+\nu}{2\lambda}x$ & 0 \\
            \hline
            case 8 & $[x,y]=0,$ $[x,z]=-x,$ $[y,z]=-y$ & The non-unimodular Lie group $G_I$ & $\left(
                                                                         \begin{array}{ccc}
                                                                           1 & 0 & 0 \\
                                                                           0 & 1 & 0 \\
                                                                           0 & 0 & \nu \\
                                                                         \end{array}
                                                                       \right)$ & $\nu>0$ & $\frac{z}{\nu}$ & 0 & $-x$ & 0 & $\frac{z}{\nu}$ & $-y$ & 0 & 0 & 0 \\
            \hline
             case 9 & $[x,y]=0,$ $[x,z]=-y,$ $[y,z]=cx-2y$ & The non-unimodular Lie group $G_c$ & $\left(
                                                                         \begin{array}{ccc}
                                                                           1 & 0 & 0 \\
                                                                           0 & \mu & 0 \\
                                                                           0 & 0 & \nu \\
                                                                         \end{array}
                                                                       \right)$ & $0<\mu\leq|c|$ , $\nu>0$ & 0 & $\frac{\mu-c}{2\nu}z$ & $\frac{c-\mu}{2\mu}y$& $\frac{\mu-c}{2\nu}z$ & $\frac{2\mu}{\nu}z$ & $\frac{c-\mu}{2}x-2y$ & $\frac{c+\mu}{2\mu}y$ & -$\frac{c+\mu}{2}x$ & 0  \\
        \hline
        case 10 & $[x,y]=0,$ $[x,z]=-y,$ $[y,z]=cx-2y$ & The non-unimodular Lie group $G_c$ & $\left(
                                                                         \begin{array}{ccc}
                                                                           1 & 0 & 0 \\
                                                                           0 & \mu & 0 \\
                                                                           0 & 0 & \nu \\
                                                                         \end{array}
                                                                       \right)$ & $\mu , \nu>0 ,$ $c=0$ & 0 & $\frac{\mu}{2\nu}z$ & -$\frac{1}{2}y$ & $\frac{\mu}{2\nu}z$ & $\frac{2\mu}{\nu}z$ & $-\frac{\mu}{2}x-2y$ & $\frac{1}{2}y$ & -$\frac{\mu}{2}x$ & 0  \\
        \hline
        case 11 & $[x,y]=0,$ $[x,z]=-y,$ $[y,z]=cx-2y$ & The non-unimodular Lie group $G_c$ & $\left(
                                                                         \begin{array}{ccc}
                                                                           1 & \frac{1}{2} & 0 \\
                                                                           \frac{1}{2} & 1 & 0 \\
                                                                           0 & 0 & \nu \\
                                                                         \end{array}
                                                                       \right)$ & $\nu>0 ,$ $c=0$ & $\frac{z}{2\nu}$ & $\frac{z}{\nu}$ & -$y$ & $\frac{z}{\nu}$ & $\frac{2z}{\nu}$ & -$2y$ & 0 & 0 & 0   \\
        \hline
        case 12 & $[x,y]=0,$ $[x,z]=-y,$ $[y,z]=cx-2y$ & The non-unimodular Lie group $G_c$ & $\left(
                                                                         \begin{array}{ccc}
                                                                           1 & 0 & 0 \\
                                                                           0 & \mu & 0 \\
                                                                           0 & 0 & \nu \\
                                                                         \end{array}
                                                                       \right)$ & $\nu>0 ,$ $c=1,$ $0<\mu\leq1$ & 0 & $\frac{\mu-1}{2\nu}z$ & $\frac{1-\mu}{2\mu}y$ & $\frac{\mu-1}{2\nu}z$ & $\frac{2\mu}{\nu}z$ & $\frac{1-\mu}{2}x-2y$ & $\frac{1+\mu}{2\mu}y$ & -$\frac{1+\mu}{2}x$ & 0   \\
        \hline
        case 13 & $[x,y]=0,$ $[x,z]=-y,$ $[y,z]=cx-2y$ & The non-unimodular Lie group $G_c$ & $\left(
                                                                         \begin{array}{ccc}
                                                                           1 & \lambda & 0 \\
                                                                           \lambda & 1 & 0 \\
                                                                           0 & 0 & \nu \\
                                                                         \end{array}
                                                                       \right)$ & $\nu>0 ,$ $c=1,$ $0<\mu\leq1,$ $0<\lambda<1$ & $\frac{\lambda}{\nu}z$ & $\frac{\lambda}{\nu}z$ & $-\frac{\lambda}{1+\lambda}(x+y)$ & $\frac{\lambda}{\nu}z$ & $\frac{2-\lambda}{\nu}z$ & $\frac{\lambda}{1+\lambda}x-\frac{2+\lambda}{1+\lambda}y$ & $-\frac{\lambda}{1+\lambda}x+\frac{1}{1+\lambda}y$ & $-\frac{1}{1+\lambda}x+\frac{\lambda}{1+\lambda}y$ & 0   \\
        \hline
        case 14 & $[x,y]=0,$ $[x,z]=-y,$ $[y,z]=cx-2y$ & The non-unimodular Lie group $G_c$ & $\left(
                                                                         \begin{array}{ccc}
                                                                           1 & 1 & 0 \\
                                                                           1 & \mu & 0 \\
                                                                           0 & 0 & \nu \\
                                                                         \end{array}
                                                                       \right)$ & $\nu>0 ,$ $c>1,$ $1<\mu\leq c$ & $\frac{z}{\nu}$ & $\frac{2+\mu-c}{2\nu}z$ & $\frac{-2+\mu+c}{2(1-\mu)}x+\frac{\mu-c}{2(1-\mu)}y$ & $\frac{2+\mu-c}{2\nu}z$ & $\frac{2\mu-c}{\nu}z$ & $\frac{-\mu(c+2)+2c+\mu^2}{2(1-\mu)}x+\frac{-2+3\mu-c}{2(1-\mu)}y$ & $\frac{-2+\mu+c}{2(1-\mu)}x+\frac{2-\mu-c}{2(1-\mu)}y$ & $\frac{\mu^2+\mu(c-2)}{2(1-\mu)}x+\frac{2-c-\mu}{2(1-\mu)}y$ & 0  \\
        \hline
         case 15 & $[x,y]=0,$ $[x,z]=-y,$ $[y,z]=cx-2y$ & The non-unimodular Lie group $G_c$ & $\left(
                                                                         \begin{array}{ccc}
                                                                           A & B & 0 \\
                                                                           B & D & 0 \\
                                                                           0 & 0 & \nu \\
                                                                         \end{array}
                                                                       \right)$ & $0\leq\mu<1$, $\nu>0$, $\lambda=\sqrt{1-c}$ & $\frac{1-\mu}{2c(1-c)\nu}z$ & $\frac{\mu-c}{2c(c-1)\nu}z$ & $\frac{-\mu}{1+\mu}x+\frac{c+c\mu-2\mu}{c(\mu^2-1)}y$ & $\frac{\mu-c}{2c(c-1)\nu}z$ & $\frac{\mu-1}{2(c-1)\nu}z$ & $\frac{c}{1+\mu}x-\frac{2+\mu}{1+\mu}y$ & $-\frac{\mu}{1+\mu}x+\frac{(-2+c+c\mu)\mu}{c(\mu^2-1)}y$ & $-\frac{c\mu}{1+\mu}x+\frac{\mu}{1+\mu}y$ & 0 \\
        \hline
        \end{tabular}
        \end{table}
\end{landscape}}
\fontsize{11}{0}
In a similar way, in Matsumoto case, it is sufficient to consider $\|\tilde{X}\|_\alpha<\frac{1}{2}$ or $|p|<\frac{1}{2\sqrt{\nu}}$. In case $11$ the Lie group $G$ admits a family of left invariant vector fields of the form $\tilde{X}=-2px+py$, $p\in{\Bbb{R}}$, which are parallel with respect to the Levi-Civita connection of $g$. Therefore, in this case, for constructing Randers metrics of Berwald type it is sufficient to let $\|\tilde{X}\|_\alpha=\sqrt{g(-2px+py,-2px+py)}=\sqrt{3}|p|<1$ or $|p|<\frac{\sqrt{3}}{3}$ and for constructing Matsumoto metrics of Berwald type it is sufficient to consider $|p|<\frac{\sqrt{3}}{6}$. By using table (\ref{Levi-Civita connection}) and a little complicated computation we can see in all other cases the Lie group $G$ does not admit a left invariant vector field $\tilde{X}$ such that it is parallel with respect to the Levi-Civita connection of $g$. Therefore proposition (\ref{existence}) guarantees that in other cases we have not any left invariant Randers or Matsumoto metric of Berwald type with the underlaying left invariant Riemannian metric $g$.
\end{proof}

Now we discuss about the flag curvatures of these metrics.

\begin{remark}
As I have shown in theorems $3.5$ and $3.8$ of \cite{Sa1} for Randers metrics (and with a similar way for Matsumoto metrics), the Finsler (Randers and Matsumoto) metrics considered in cases $(i)$ and $(ii)$ of theorem (\ref{Classification}) are flat geodesically complete locally Minkowskian Finsler metrics.
\end{remark}

Therefore it is sufficient to compute the flag curvature of Randers and Matsumoto metrics of case $(iii)$.

\begin{theorem}
Let $F=\alpha+\beta$ be the left invariant Randers metric derived from case $(iii)$ of theorem (\ref{Classification}). Suppose that $\{U=ax+by+cz,V=\tilde{a}x+\tilde{b}y+\tilde{c}z\}$ is an orthonormal set with respect to $g$. Then the flag curvature $K(P,U)$ of the flag $(P=span\{U,V\},U)$ is given by
\begin{equation}\label{flag of Randers}
    K(P,U)=-4(\frac{c(\tilde{a}+2\tilde{b})-\tilde{c}(a+2b)}{3pa-2})^2.
\end{equation}
Therefore $F$ is of non-positive flag curvature.
\end{theorem}
\begin{proof}
$F$ is of Berwald type so the curvature tensors of $g$ and $F$ coincide. By using table (\ref{Levi-Civita connection}) we can compute the curvature tensor as follows
\begin{eqnarray}
  R(x,z)x &=& -R(z,x)x =\frac{1}{2}R(x,z)y=-\frac{1}{2}R(z,x)y=\frac{1}{2}R(y,z)x=-\frac{1}{2}R(z,y)x\\
  &&=\frac{1}{4}R(y,z)y=-\frac{1}{4}R(z,y)y=\frac{z}{\nu},\nonumber \\
 -\frac{1}{2}R(x,z)z&=& \frac{1}{2}R(z,x)z=-\frac{1}{4}R(y,z)z=\frac{1}{4}R(z,y)z=y,\nonumber
 \end{eqnarray}
 $R=0$, for other cases.
 Therefore we have
 \begin{equation}\label{R(V,U)U}
    R(V,U)U=\frac{1}{\nu}\{(a+2b)(\tilde{a}c-\tilde{c}a+2(\tilde{b}c-\tilde{c}b))\}z-2c\{\tilde{a}c-\tilde{c}a+2(\tilde{b}c-\tilde{c}b)\}y.
 \end{equation}
 Let $\delta:=\frac{1}{\nu}\{(a+2b)(\tilde{a}c-\tilde{c}a+2(\tilde{b}c-\tilde{c}b))\}$ and $\sigma:=-2c\{\tilde{a}c-\tilde{c}a+2(\tilde{b}c-\tilde{c}b)\}$. Now by using formula of $g_U$ given in \cite{EsSa} we have
 \begin{eqnarray}
            g_U(R(V,U)U,V) &=&  (\delta\tilde{c}\nu+\sigma(\tilde{b}+\frac{\tilde{a}}{2}))(-\frac{3}{2}pa+1)-\frac{3}{2}p\tilde{a}(\delta c\nu+\sigma(b+\frac{a}{2}))\\
              g_U(U,U) &=& (-\frac{3}{2}pa+1)^2 \\
              g_U(V,V) &=& 1-\frac{3}{2}pa+(-\frac{3}{2}p\tilde{a})^2 \\
              g_U(U,V) &=& -\frac{3}{2}p\tilde{a}(-\frac{3}{2}pa+1).
 \end{eqnarray}
Substitution of the last equations in (\ref{flag curvature}) will completes the proof.
\end{proof}

\begin{theorem}
Suppose that $F=\frac{\alpha^2}{\alpha-\beta}$ is the left invariant Matsumoto metric derived from case $(iii)$ of theorem (\ref{Classification}). Then, with the assumptions of the previous theorem, the formula for the flag curvature of $F$ is as follows

\begin{equation}\label{flag of Matsumoto}
    K(P,U)=-\frac{(2+3ap)^3(1+3ap)(\tilde{c}a-\tilde{a}c+2(\tilde{c}b-\tilde{b}c))^2}{2(4+18a^2p^2+18ap-27\tilde{a}^2p^2)}.
\end{equation}

\begin{proof}
The proof is similar to the previous theorem except that for considered Matsumoto metric by a direct computation we have
\begin{eqnarray}
            g_U(R(V,U)U,V) &=&  \frac{1}{(1+\frac{3}{2}ap)^4}\{(\sigma(\frac{\tilde{a}}{2}+\tilde{b})+\tilde{c}\delta\nu)(1+3ap)(1+\frac{3}{2}ap)\\
            &&-\frac{3}{2}\tilde{a}p(\sigma(\frac{a}{2}+b)+c\delta\nu)(1+6ap)\}\nonumber\\
              g_U(U,U) &=&  \frac{1}{(1+\frac{3}{2}ap)^2}\\
              g_U(V,V) &=&  \frac{1+\frac{9}{2}(a^2p^2-\tilde{a}^2p^2+ap)}{(1+\frac{3}{2}ap)^4}\\
              g_U(U,V) &=& \frac{-3\tilde{a}p}{2(1+\frac{3}{2}ap)^3}.
 \end{eqnarray}

\end{proof}

\end{theorem}

\bibliographystyle{amsplain}

\end{document}